\documentclass[a4paper,12pt]{amsart}

 \usepackage{amssymb,amsthm}
\usepackage{amscd} 
\usepackage{graphicx,color} 

\usepackage{url} 

\newtheorem{theorem}{Theorem}[section]

\newtheorem{lemma}[theorem]{Lemma}

\theoremstyle{definition}

\newtheorem{remark}[theorem]{Remark}

\def\r{\mathbb R}

\begin{document}

\title[Rigidity of solitons of the Gauss curvature flow in Euclidean space]{Rigidity of solitons of the Gauss curvature flow in Euclidean space}
\author{Rafael L\'opez}
\address{ Departamento de Geometr\'{\i}a y Topolog\'{\i}a\\ Universidad de Granada. 18071 Granada, Spain}
\email{rcamino@ugr.es}


\begin{abstract}  In this paper, we consider $\lambda$-translating solitons and $\lambda$-shrinkers of the Gauss curvature flow in Euclidean space. We prove that planes and circular cylinders are the only $\lambda$-translating solitons with constant mean curvature. We also prove that planes, spheres and circular cylinders are the only $\lambda$-shrinkers  with constant mean curvature. We give a classification of the $\lambda$-translating solitons and $\lambda$-shrinkers with one constant principal curvature.
\end{abstract}

\maketitle

\section{Introduction and statement of the results}

 Let $\Sigma$ be a surface and let $\Phi:\Sigma\to\r^3$ be an    immersion  in Euclidean space $\r^3$. The Gauss curvature flow (GCF) is a one-parameter family of smooth   immersions $\Phi_t=\Phi(\cdot,t)\colon\Sigma\to \r^3$, $t\in [0,T)$ such that $\Phi_0=\Phi$ and satisfying   
$$\frac{\partial}{\partial t}\Phi(p,t)=-K(p,t)  N(p,t),\quad (p,t)\in \Sigma\times [0,T),$$
where   $N(p,t)$ is the unit normal of $\Phi(p,t)$ and $K(p,t)$ is the Gauss curvature at $\Phi(p,t)$ \cite{an1,ch,ur}.  In the study of the GCF, it is of special interest  the possible singularities that can appear \cite{an2,br}. These singularities are studied by means of solitons of the GCF. There are two   types of solitons. For its definition,  we consider a more general context. Let  $\lambda\in\r$ and $\vec{v}\in\r^3$  be a unit vector. The surface $\Sigma$ is said to be a {\it $\lambda$-translating soliton} of the  GCF if  
\begin{equation}\label{eq1}
K=\langle N,\vec{v}\rangle+\lambda.
\end{equation} 
The second type of solitons is  a {\it $\lambda$-shrinker}. Let $\alpha\in\r$, $\alpha\not=0$. A surface $\Sigma$ is said to be a $\lambda$-shrinker of the GCF if 
\begin{equation}\label{eq2}
K=\alpha\langle N,\Phi\rangle+\lambda.
\end{equation} 
  Examples of  solitons of the GCF are the following:
 \begin{enumerate}
 \item Planes. The   Gauss curvature is  $K=0$. If $N=\vec{w}$ is the unit normal vector of the plane, then the plane is a $\lambda$-translating soliton for $\lambda=-\langle\vec{w},\vec{v}\rangle$. Vector planes, that is, planes crossing the origin of $\r^3$, are   $\lambda$-shrinkers for $\lambda=0$ and for all $\alpha\in\r$.
 \item Spheres. A sphere of radius $r>0$ centered at the origin  is a $\lambda$-shrinker for all $\lambda,\alpha\in\r$ satisfying $\lambda r^2-\alpha r^3-1=0$.   Spheres are not $\lambda$-translating solitons.
 \item Circular cylinders. If   $\vec{v}$ is parallel to the axis of the cylinder, and since $K=0$, then the cylinder is a $\lambda$-translating soliton  for $\lambda=0$. If the radius is $r>0$ and if we take the inward orientation on the cylinder, then it is a $\lambda$-shrinker for all $\alpha,\lambda$ such that $\lambda-\alpha r=0$
  \end{enumerate}

Similarly,  there is   a theory on the mean curvature flow (MCF) as well as the study of the corresponding solitons. Solitons are  the so-called translating  solitons and shrinkers, where the mean curvature $H$ satisfies  $H=\langle N,\vec{v}\rangle+\lambda$ and $H=\langle N,\Phi\rangle+\lambda$, respectively. For $\lambda=0$, these solitons are models of  type II singularities under MCF  \cite{hu}. In the literature of the  MCF theory, there is a set of results of `rigidity type' that characterize those solitons with constant Gauss curvature or assuming some Weingarten relation. For example, if $\lambda=0$, the only translating solitons with  $K=0$ are planes and  grim reapers \cite{ms}. It was proved in  \cite{li}  that there are no complete $\lambda$-translating solitons with nonzero constant Gauss.   If it is assumed a linear Weingarten relation between its curvatures,  then the only translating solitons are planes, circular cylinders and   a  certain type of cylindrical surfaces \cite{flly}. For  shrinkers, similar results have been obtained \cite{co,fu2,yf,yfl}.

Recently, in \cite{fl} the authors have studied surfaces in Euclidean space satisfying the equation $K^\alpha=\langle N,\vec{v}\rangle$. Among their results, they have proved that if the mean curvature is constant, then the surface is a plane or a circular cylinder. Notice that this equation coincides with \eqref{eq1} when $\lambda=0$ and $\alpha=1$. 

In this paper we extend this type of results for solitons of the CGF. As hypothesis, it is  assumed that the mean curvature $H$ is constant.   For the two types of solitons of the GCF, we have the following rigidity results.

\begin{theorem}\label{t1}
Planes and circular cylinders are the only   $\lambda$-translating solitons of the GCF with constant mean curvature.
 \end{theorem}

\begin{theorem}\label{t2}
Planes, spheres and circular cylinders are the only  $\lambda$-shrinkers of the GCF  with constant mean curvature. \end{theorem}

Theorem \ref{t1} is proved in Sect.   \ref{s3}, while Thm. \ref{t2} is proved in Sect. \ref{s6}. Besides these results, we also classify the solitons with one constant principal curvature.  For the case of $\lambda$-translating solitons, we prove that these surfaces are planes, circular cylinders and a family of surfaces that make constant angle with the direction $\vec{v}$: see Thm. \ref{t-c} in Sect. \ref{s3}. For $\lambda$-shrinkers, it is proved that if a principal curvature is constant, then the surface is a plane, a sphere or a circular cylinder: see Thm. \ref{t-d} in Sect. \ref{s5}.  

In opinion of the author of the present paper, with a  more effort, such results may be generalized replacing the hypothesis of the constancy of the mean curvature by other assumptions, as for example, a certain Weingarten relation between the principal curvatures or between $K$ and $H$. Notice that similar results can be found in \cite{fl,flly,fu2}. 

\section{Preliminaries}\label{s2}

In this section we recall the coordinates of a surface given by lines of curvatures. These coordinates are defined on open sets of a surface free of umbilic points. The following computations are common in the discussions for both types of solitons of the GCF. 

Let $\Sigma$ be an orientable surface immersed in $\r^3$. Denote by  $\nabla$ and $\overline{\nabla}$ be the Levi-Civita connections of $\Sigma$ and
$\mathbb{R}^3$, respectively.  Denote by $N$ the Gauss map of $\Sigma$. Let $\mathfrak{X}(\Sigma)$ be the space of tangent vector fields of $\Sigma$. The  Gauss and Weingarten formulae are, respectively, 
\begin{equation*}  
\begin{split}
    \overline{\nabla}_XY&=\nabla_XY+h(X,Y), \\
    \overline{\nabla}_XN&=-SX
    \end{split}
\end{equation*}
for all $X,Y\in\mathfrak{X}(\Sigma)$, where $h\colon T\Sigma\times T\Sigma\to T^\perp\Sigma$ and  $S\colon T\Sigma\to  T\Sigma$ are the second fundamental form and the shape
operator, respectively. The  Gauss
equation and Codazzi equations,  
\begin{equation*}
\begin{split}
    \langle R(X,Y)W,Z\rangle&=\langle h(Y,W),h(X,Z)\rangle-\langle
    h(X,W),h(Y,Z)\rangle\\
    (\nabla_XS)Y&=(\nabla_YS)X,
    \end{split}
\end{equation*}
where $R$ is the curvature tensor of $\nabla$. The   mean curvature vector $\vec{H}$ is defined by
\[     \vec{H}=\frac12{\rm trace}\ h=HN,\]
where $H$ is the mean curvature of $\Sigma$. If $\kappa_1$ and $\kappa_2$ are the principal curvatures of $\Sigma$, then $H=\frac{\kappa_1+\kappa_2}{2}$ and the Gauss curvature is $K=\kappa_1\kappa_2$.

In an open set of $\Sigma$ of non-umbilical points, consider a local
orthonormal frame $\{e_1,e_2\}$ given by line of curvature coordinates. The expression of the shape operator $S$ is 
\begin{equation}\label{b0}
\left\{
\begin{split}
Se_1&=\kappa_1 e_1\\
Se_2&=\kappa_2 e_2.
\end{split}
\right.
    \end{equation}
Since $\{e_1,e_2\}$ is an orthonormal frame, if $X\in\mathfrak{X}(\Sigma)$ then
\begin{equation*}
\begin{split}
\nabla_X e_1&=\omega(X)e_2\\
 \nabla_X e_2&=-\omega(X)e_1,
\end{split}
\end{equation*}
where $\omega(X)=\langle\nabla_X e_1,e_2\rangle=-\langle\nabla_X e_2,e_1\rangle$. Then we have 
\begin{equation}\label{eiej}
\left\{
\begin{split}
&\nabla_{e_1}e_1=\omega(e_1)e_2\quad \nabla_{e_1}e_2=-\omega(e_1)e_1\\
&\nabla_{e_2}e_1=\omega(e_2)e_2\quad \nabla_{e_2}e_2=-\omega(e_2)e_1.
\end{split}\right.
\end{equation}
Using the Codazzi equations, we have  
\begin{equation}\label{b1}
    \left\{
    \begin{array}{ll}
        e_{2}(\kappa_1)=(\kappa_1-\kappa_2)\omega(e_1)\\
        e_{1}(\kappa_2)=(\kappa_1-\kappa_2)\omega(e_2).
    \end{array}
    \right.
\end{equation}
In the set $\Omega$, the  Gauss equation  implies the following expression for $K$,  
\begin{equation}\label{bb}
   K=\kappa_1\kappa_2=-e_1\left(\frac{e_1( \kappa_2 )}{\kappa_1-\kappa_2} \right)+e_2\left(\frac{e_2( \kappa_1 )}{\kappa_1 -\kappa_2}\right)
   -\frac{e_1( \kappa_2)^2+ e_2( \kappa_1 )^2}{(\kappa_1 -\kappa_2)^2}.
\end{equation}

\section{Translating solitons: a principal curvature is constant}\label{s3}

 In this section, we classify all $\lambda$-translating solitons where  one principal curvature is constant. 
 
Let $\Sigma$ be a  $\lambda$-translating soliton of the GCF.  For   the unit vector $\vec{v}\in\mathbb{R}^3$ in \eqref{eq1}, consider the decomposition of $\vec{v}$ in its tangent and normal part with respect to $\Sigma$, namely, 
$$\vec{v}=\vec{v}^\top+\vec{v}^\perp=\vec{v}^\top+ \langle\vec{v},N\rangle N.$$
If $\Sigma$ satisfies \eqref{eq1}, then  
$$\vec{v}^\perp=(K-\lambda)N,\qquad \mbox{on $\Sigma$.}$$

If the interior of the set of umbilic points is not empty, then $\Sigma$ is umbilical in an open set. Since spheres do not satisfy \eqref{eq1}, then this open set is contained in a plane. Thus, if $\Sigma$ does not contain open sets of planes, the interior  of the set of umbilic points is an empty set. Let $p\in \Sigma$ be a non-umbilic point  and we work in open set $\Omega\subset \Sigma$ containing $p$, consisting entirely of non-umbilic points. In $\Omega$, consider line of curvature coordinates $\{e_1,e_2\}$ as we   described in Sect. \ref{s2}.

Decompose the   tangent part $\vec{v}^\top$ in coordinates   with respect to $\{e_1,e_2\}$ by
\begin{equation}\label{vtop}
\vec{v}^{\top}=\gamma e_{1}+\mu e_{2},
\end{equation}
where $\gamma$ and $\mu$ are smooth functions on $\Omega$. 

\begin{lemma}\label{l1}
Let $\Omega$ be a subset of non-umbilical points of a $\lambda$-translating soliton of the GCF. Then the functions   $\gamma$ and $\mu$ satisfy the following equations
    \begin{equation}\label{b3}
    \left\{
\begin{split}
            &e_{1}(\gamma)-\mu\omega(e_1)-(K-\lambda)\kappa_1=0\\
            &e_{2}(\gamma)-\mu\omega(e_2)=0\\
            &e_{1}(\mu)+\gamma\omega(e_1)=0\\
            &e_{2}(\mu)+\gamma\omega(e_2)-(K-\lambda)\kappa_2=0\\
            &e_{1}(K)+\gamma\kappa_1=0\\
            &e_{2}(K)+\mu\kappa_2=0.
        \end{split}\right.
    \end{equation}
\end{lemma}

\begin{proof} From $\vec{v}^{\perp}=(K-\lambda)N$, the Gauss and
Weingarten formulas imply 
\begin{equation*}
\begin{split}
    0&=\overline{\nabla}_{X}\vec{v}=\overline{\nabla}_{X}(\vec{v}^{\top}+\vec{v}^{\perp})\\
    &={\nabla}_{X}\vec{v}^{\top}+h(X,\vec{v}^{\top})+X(K)N-(K-\lambda)SX 
\end{split}
\end{equation*}
for all   $X\in\mathfrak{X}(\Sigma)$.
Thus,  the tangent and the normal parts in this identity give
\begin{equation*} 
    \left\{
    \begin{array}{ll}
        {\nabla}_{X}\vec{v}^{\top}-(K-\lambda)SX=0\\
        h(X, \vec{v}^{\top})+X(K)N=0.
    \end{array}
    \right.
\end{equation*}
In both equations, we substitute $X$ by $e_1$ and   $e_2$. Then \eqref{b3} is a straightforward consequence of identifying the coordinates with respect to the basis $\{e_1,e_2\}$. We show one example of that. Taking $X=e_1$, we have that the tangent part is 
$\nabla_{e_1}(\gamma e_1+\mu e_2)-(K-\lambda)Se_1=0$. This gives 
$$e_1(\gamma)e_1+e_1(\mu)e_2-(K-\lambda)\kappa_1e_1+\gamma\omega(e_1)e_2-\mu\omega(e_1)e_1=0.$$
This gives the first and the third equation of \eqref{b3}.
\end{proof}

From the previous computations, we give the first result of classification of $\lambda$-translating solitons of the GCF assuming that one of the two principal curvatures of the surface is constant.   In  this context, it will appear   the surfaces in $\r^3$ that make a constant angle with a fixed direction of $\r^3$: see \cite{mn} for its classification.

\begin{theorem}\label{t-c} Let $\Sigma$ be a $\lambda$-translating soliton of the GCF. If a principal curvature is constant, then $\Sigma$ is one of the following surfaces:
\begin{enumerate}
\item a plane;
\item a cylindrical surface $C\times\r\vec{v}$, where $C\subset\r^3$ is a curve contained in a plane orthogonal to $\vec{v}$;
\item  $K=0$, $0<|\lambda|<1$ and $\Sigma$ is a surface which makes a constant angle $\theta$ with $\vec{v}$, with $\cos\theta=|\lambda|$.
\end{enumerate}
\end{theorem}

\begin{proof} Suppose that the principal curvature $\kappa_2$ is constant. If the   principal curvature $\kappa_1$ is also constant, then the surface is isoparametric. This implies  that $\Sigma$ is a plane, a sphere or a circular cylinder but sphere is not a translating soliton. This proves that the surface is a plane or a circular cylinder. Both examples are included in the cases (1) and (2) of the theorem, respectively. 

From now, we discard the case that $\kappa_1$ is constant. We distinguish  the cases $\kappa_2=0$ and $\kappa_2\not=0$. 
\begin{enumerate}
\item Case $\kappa_2=0$. Then $K=0$. Since $N$ and $\vec{v}$ are unit vectors, we have that $|\lambda|\leq 1$. Then we have $\langle N,\vec{v}\rangle=-\lambda$. This implies that $\Sigma$ is a surface that makes a constant angle with the direction $\vec{v}$. These surfaces were classified in \cite{mn}. For example, if $|\lambda|=1$, then $\Sigma$ is a plane orthogonal to $\vec{v}$. If $|\lambda|=0$, then $\Sigma$ is a cylindrical surface $C\times \r\vec{v}$ where $C$ is a curve contained in a plane orthogonal to $\vec{v}$. 

\item Case $\kappa_2\not=0$.    Let us work on the  open set $\Omega\subset \Sigma$ formed by non-umbilic points. From \eqref{b1} we have $\omega(e_2)=0$. By using  the last and the fourth expressions of \eqref{b3}, we deduce   $\mu=-e_2(\kappa_1)$ and $e_2(\mu)=(K-\lambda)\kappa_2$, respectively. From the first identity, we have  
$e_2(\mu)=-e_{22}(\kappa_1)$, where  $e_{ij}(\kappa_1)=e_ie_j(\kappa_1)$ for all $i,j$. Hence 
\begin{equation}\label{e22}
e_{22}(\kappa_1)=-(K-\lambda)\kappa_2.
\end{equation}
Using this identity, and because $\kappa_2$ is constant, equation \eqref{bb} is 
$$\kappa_1\kappa_2=\frac{e_{22}(\kappa_1)}{\kappa_1-\kappa_2}-2\frac{e_2(\kappa_1)^2}{(\kappa_1-\kappa_2)^2}=
-\frac{(K-\lambda)\kappa_2}{\kappa_1-\kappa_2}-2\frac{e_2(\kappa_1)^2}{(\kappa_1-\kappa_2)^2}.$$
 Then 
 $$2e_2(\kappa_1)^2=-(K-\lambda) \kappa_2(\kappa_1-\kappa_2) -\kappa_1\kappa_2 (\kappa_1-\kappa_2)^2.$$
 Differentiating with respect to $e_2$, we have 
\begin{equation}\label{fp2}
4e_2(\kappa_1)e_{22}(\kappa_1)=-e_2(\kappa_1)\kappa_2\left( (\kappa_2+2\kappa_1)(\kappa_1-\kappa_2)+ K-\lambda + (\kappa_1-\kappa_2)^2\right).
\end{equation}
 The function  $e_2(\kappa_1)$ is not zero identically because otherwise, we have $\mu=0$ and the forth equation in \eqref{b3} yields $K-\lambda=0$. This implies that both principal curvatures are constant which it is not possible. 
  
 We now insert  \eqref{e22} in \eqref{fp2}, obtaining a polynomial on the function $\kappa_1$, namely, 
 $3\kappa_1^2-6\kappa_2\kappa_1-5\lambda=0$. This gives a contradiction because $\kappa_1$ is a non-constant function. We conclude that the case $\kappa_2\not=0$ cannot occur.
 \end{enumerate}
  \end{proof}
  
 \begin{remark} Surfaces in Euclidean space that make a constant angle with a fixed direction are classified in \cite{mn}. These surfaces have zero Gauss curvature $K=0$, in particular, they are $\lambda$-translating solitons of the GCF if $|\lambda|\leq 1$. 
 \end{remark}
 
 \begin{remark} Cylindrical surfaces $C\times\r\vec{v}$ are examples of $\lambda$-translating solitons for $K=0$ and $\lambda=0$. This family of surfaces includes circular cylinders but there are more surfaces. Here we point out a gap in Theorem 1.2 of \cite{fl}. In that result, the authors do not obtain the surfaces of type $C\times\r\vec{v}$, which satisfy $K^\alpha=\langle N,\vec{v}\rangle$ for $K=0$ and $\vec{v}$ parallel to the rulings of the surfaces. This is because in the proof, the authors assume $\kappa_2=0$ but the principal curvature $\kappa_1$ (the curvature of the curve $C$) is, in general, a function.
 \end{remark}
\section{Proof of Theorem \ref{t1} }\label{s4}
 
Let $\Sigma$ be a $\lambda$-translating soliton with constant mean curvature $H$. Let  $2H=\kappa_1+\kappa_2$.  If a principal curvature is constant, then the other one is also constant because $H$ is constant. Then $\Sigma$ is a isoparametric surface hence it  is a plane, a sphere or a circular cylinder. However, spheres are not $\lambda$-translating solitons. This proves the theorem for this situation. 

From now on, we will assume that neither  $\kappa_1$ nor $\kappa_2$ are constant and we will arrive to a contradiction. Let $p\in \Sigma$ be a non-umbilic point and let   $\Omega\subset \Sigma$ be an open set of $p$ formed by non-umbilic points. We follow the same notation of Sect.   \ref{s2}. The main idea is to express all expressions in terms of $\kappa_1$ and $H$ and next, use that $H$ is a constant function. First, we have
\begin{equation}\label{case1}
\begin{split}
\kappa_2&=2H-\kappa_1,\\
 K&=\kappa_1(2H-\kappa_1),\\		
\kappa_1-\kappa_2&=2(\kappa_1-H).
\end{split}
\end{equation}
In a subdomain of $\Omega$, if necessary, we can assume $\kappa_1\not=0$ and $\kappa_1-H\not=0$ because the principal curvatures are not constant. 

We work with Eq.  \eqref{bb}. From \eqref{case1}, we have $e_1(\kappa_2)=-e_1(\kappa_1)$ and $e_2(\kappa_2)=-e_2(\kappa_1)$. Then the identity \eqref{bb} is now
\begin{equation}\label{c11}
\kappa_1(2H-\kappa_1)=\frac{e_{11}(\kappa_1)+e_{22}(\kappa_1)}{2(\kappa_1-H)}-\frac{3(e_1(\kappa_1)^2+e_2(\kappa_1)^2)}{4(\kappa_1-H)^2}.
\end{equation}
In \eqref{case1}, we differentiate $K$  with respect to $e_1$ and $e_2$, obtaining
\begin{equation*}
\left\{
\begin{split}
e_1(K)&=-2( \kappa_1-H)e_1(\kappa_1),\\
e_2(K)&=-2(\kappa_1-H)e_2(\kappa_1).
\end{split}
\right.
\end{equation*}
From \eqref{b3}, we obtain 
\begin{equation*}
\begin{split}
0&=2(H-\kappa_1)e_1(\kappa_1)+\gamma\kappa_1,\\
0&=2(H-\kappa_1)e_2(\kappa_1)+\mu(2H-\kappa_1).
\end{split}
\end{equation*}
This gives 
\begin{equation}\label{c1}
\gamma=\frac{2(\kappa_1-H)e_1(\kappa_1)}{\kappa_1},\quad \mu=\frac{2(\kappa_1-H)e_2(\kappa_1)}{2H-\kappa_1}.
\end{equation}
From \eqref{b1}, we obtain 
\begin{equation}\label{c2}
\omega(e_1)=\frac{e_2(\kappa_1)}{2(\kappa_1-H)},\quad \omega(e_2)=-\frac{e_1(\kappa_1)}{2(\kappa_1-H)}.
\end{equation}
Substituting in  \eqref{b3}, and using the expressions of $\gamma$ and $\mu$ given in \eqref{c1}, we have 
\begin{equation}\label{12}
\left\{
\begin{split}
e_1(\gamma)&=\frac{e_2(\kappa_1)^2}{2H-\kappa_1}+(K-\lambda)\kappa_1\\
e_1(\mu)&=-\frac{e_1(\kappa_1)e_2(\kappa_1)}{\kappa_1}\\
e_2(\gamma)&=-\frac{e_1(\kappa_1)e_2(\kappa_1)}{2H-\kappa_1}\\
e_2(\mu)&=\frac{e_1(\kappa_1)^2}{\kappa_1}+(K-\lambda)(2H-\kappa_1).
\end{split}\right.
\end{equation}

On the other hand,   differentiating \eqref{c1} with respect to $e_1$ and $e_2$, we obtain
\begin{equation}\label{13}
\left\{
\begin{split}
e_1(\gamma)&=\frac{2(\kappa_1-H)}{\kappa_1}e_{11}(\kappa_1)+\frac{2H}{\kappa_1^2}e_1(\kappa_1)^2\\
e_1(\mu)&=\frac{2(\kappa_1-H)}{2H-\kappa_1}e_{12}(\kappa_1)+\frac{2H}{(2H-\kappa_1)^2}e_1(\kappa_1)e_2(\kappa_1)\\
e_2(\gamma)&=\frac{2(\kappa_1-H)}{\kappa_1}e_{21}(\kappa_1)+\frac{2H}{\kappa_1^2}e_1(\kappa_1)e_2(\kappa_1)\\
e_2(\mu)&=\frac{2(\kappa_1-H)}{2H-\kappa_1}e_{22}(\kappa_1)+\frac{2H}{(2H-\kappa_1)^2}e_2(\kappa_1)^2.
\end{split}\right.
\end{equation}

From \eqref{12} and \eqref{13}, we obtain the expressions of $e_{ij}(\kappa_1)$, namely, 
\begin{subequations}
 \begin{align}
 e_{11}(\kappa_1)&=-\frac{He_1(\kappa_1)^2}{\kappa_1(\kappa_1-H)}+\frac{\kappa_1 e_2(\kappa_1)^2}{2(\kappa_1-H)(2H-\kappa_1)}+\frac{(K-\lambda)\kappa_1^2}{2(\kappa_1-H)},\label{c7}\\
e_{12}(\kappa_1)&=-\frac{\kappa_1^2-2H\kappa_1+4H^2}{2\kappa_1(\kappa_1-H)(2H-\kappa_1)}e_1(\kappa_1)e_2(\kappa_1),\label{c9}\\
e_{22}(\kappa_1)&=\frac{(2H-\kappa_1)e_1(\kappa_1)^2}{2\kappa_1(\kappa_1-H)}-\frac{He_2(\kappa_1)^2}{(\kappa_1-H)(2H-\kappa_1)}+\frac{(K-\lambda)(2H-\kappa_1)^2}{2(\kappa_1-H)}.\label{c10}
\end{align}
\end{subequations}
Since $\kappa_1$ cannot be constant in an open set of $\Sigma$,   we can assume without loss of generality  that $e_2(\kappa_1)\not=0$ in $\Omega$. Looking Eq. \eqref{c7}, we have two possibilities.

\begin{enumerate} 
\item Case  $e_1(\kappa_1)=0$ identically in $\Omega$. Then \eqref{c7}  implies
\begin{equation}\label{e77}
e_2(\kappa_1)^2=- \kappa_1(2H-\kappa_1)(K-\lambda).
\end{equation}
Then \eqref{c10} is 
\begin{equation}\label{e88}
e_{22}(\kappa_1)=\frac{K-\lambda}{2(\kappa_1-H)}( \kappa_1^2-2H\kappa_1+4H^2).
\end{equation}
Differentiating Eq. \eqref{e77} with respect to $e_2$,  and because $e_2(\kappa_1)\not=0$, we obtain 
\begin{equation}\label{e99}
e_{22}(\kappa_1)=(\kappa_1-H)(2K-\lambda). 
\end{equation}
Comparing \eqref{e88} and \eqref{e99}, we deduce
$$(\kappa_1-H)(2H-\lambda)=\frac{K-\lambda}{2(\kappa_1-H)}(\kappa_1^2-2H\kappa_1+4H^2).$$
This gives a polynomial on $\kappa_1$ with constant coefficients of degree $4$, namely, 
$$3\kappa_1^4-12H \kappa_1^3+2(6H^2+\lambda)\kappa_1^2-4H\lambda \kappa_1+2H^2\lambda=0.$$
 This   implies that $\kappa_1$ is constant, obtaining a contradiction.

\item Case  $e_1(\kappa_1)\not=0$   in $\Omega$. Substituting \eqref{c7} and \eqref{c10} into \eqref{c11}, we have
\begin{equation}\label{c14}
    2e_1(\kappa_1)^2+2e_2(\kappa_1)^2+R_1=0,
\end{equation}
where
$$R_1=(2H^2-2H\kappa+\kappa^2)\lambda-\kappa^2(2H-\kappa)^2.$$
Differentiating \eqref{c14} with respect to $e_1$, we have
\begin{equation}\label{c15}
4e_1(\kappa_1)e_{11}(\kappa_1)+4e_2(\kappa_1)e_{12}(\kappa_1)+e_1(R_1)=0.
\end{equation}
By using \eqref{c7} and \eqref{c9}, and simplifying by $e_1(\kappa_1)$,  equation \eqref{c15} can be rewritten by
\begin{equation}\label{c16}
2H e_1(\kappa_1)^2+2H e_2(\kappa_1)^2+R_2=0,
\end{equation}
where
$$R_2=\kappa_1\lambda  \left(2 H^2-4 H \kappa_1+3 \kappa_1^2\right)+\kappa_1^2 (2 H-3 \kappa_1) (2H-\kappa_1)^2.$$
Combining \eqref{c14}  and \eqref{c16}, we obtain  
$$
\left(2 H^2-3 H \kappa_1+\kappa_1^2\right) \left(\lambda  \left(-2 H^2+2 H \kappa_1-3 \kappa_1^2\right)+3 \kappa_1^2 (\kappa_1-2 H)^2\right)=0.$$ 
This is a polynomial equation on $\kappa_1$ with constant coefficients, deducing that $\kappa_1$ is constant. This       contradiction finishes the proof of  Theorem \ref{t1}.

\end{enumerate}
 
As a consequence of the above computations, we get the classification of $\lambda$-translating solitons of the GCF with constant curvature. If $K$ is constant, it is immediate that $\langle N,\vec{v}\rangle=K-\lambda$ and the surface   falls in the family of surfaces of Euclidean space $\r^3$ making a constant angle with a fix direction \cite{mn}. In particular, $K=0$.

\begin{theorem} Surfaces of Thm. \ref{t-c} are the only    $\lambda$-translating solitons of the GCF with constant curvature $K$. In particular, $K=0$. 
\end{theorem}

\begin{proof}
 From \eqref{b3}, we have   $\gamma\kappa_1=0$ and $\mu\kappa_2=0$. If $\gamma=\mu=0$ identically, then $\vec{v}^\top$ and the surface is a plane orthogonal to $\vec{v}$. Otherwise, $K=0$. Then $\langle N,\vec{v}\rangle=-\lambda$ and the result follows from Thm. \ref{t-c}. 
\end{proof}

\section{$\lambda$-shrinkers: a principal curvature is constant}\label{s5}

In this and in the following sections, we study $\lambda$-shrinkers of the GCF. In this section, we classify those shrinkers with one constant principal curvature.  Denote by $\Phi\colon\Sigma\to\r^3$ the immersion of the $\lambda$-shrinker. Let  decompose the position vector $\Phi$ in its tangent and normal part with respect to $\Sigma$,  
\begin{equation*}
\Phi^\top=\gamma e_1+\mu e_2,
\end{equation*}
for some smooth functions $\gamma$ and $\mu$. The equivalent  result of Lem. \ref{l1} is the following. 
 
\begin{lemma}\label{l2}
Let $\Omega$ be a subset of non-umbilical points of a $\lambda$-translating soliton of the GCF. Then the functions   $\gamma$ and $\mu$ satisfy the following equations
    \begin{equation}\label{d3}
    \left\{
\begin{split}
            &e_{1}(\gamma)-\mu\omega(e_1)-\frac{K-\lambda}{\alpha}\kappa_1=1\\
            &e_{2}(\gamma)-\mu\omega(e_2)=0\\
            &e_{1}(\mu)+\gamma\omega(e_1)=0\\
            &e_{2}(\mu)+\gamma\omega(e_2)-\frac{K-\lambda}{\alpha}\kappa_2=1\\
            &e_{1}(\frac{K}{\alpha})+\gamma\kappa_1=0\\
            &e_{2}(\frac{K}{\alpha})+\mu\kappa_2=0.
        \end{split}\right.
    \end{equation}
\end{lemma}

\begin{proof} The normal part of $\Phi$ is   $\Phi^{\perp}=\frac{K-\lambda}{\alpha}N$. If $X\in\mathfrak{X}(\Sigma)$, then   $\nabla_X\Phi=X$. Then the Gauss and
Weingarten formulas imply 
\begin{equation*}
\begin{split}
    X&=\overline{\nabla}_{X}\Phi=\overline{\nabla}_{X}(\Phi^{\top}+\Phi^{\perp})\\
    &={\nabla}_{X}\Phi^{\top}+h(X,\Phi^{\top})+X(\frac{K}{\alpha})N-\frac{K-\lambda}{\alpha}SX 
\end{split}
\end{equation*}
for all   $X\in\mathfrak{X}(\Sigma)$.
The difference appears in the tangent part of the above expression, obtaining  
\begin{equation*} 
    \left\{
    \begin{array}{ll}
        {\nabla}_{X}\Phi^{\top}-\frac{K-\lambda}{\alpha}SX=X\\
        h(X, \Phi^{\top})+X(\frac{K}{\alpha})N=0.
    \end{array}
    \right.
\end{equation*}
Equations \eqref{d3} are obtained by substituting $X$ by $e_1$ and   $e_2$. and writing in coordinates with respect to $\{e_1,e_2,N\}$. 
\end{proof}
 
As in the case of shrinkers, we classify the $\lambda$-shrinkers with one constant principal curvature.  
 
\begin{theorem}\label{t-d} Let $\Sigma$ be a $\lambda$-shrinker of the GCF. If a principal curvature is constant, then $\Sigma$ is a plane, a sphere or a circular cylinder. 
\end{theorem}

\begin{proof} Without loss of generality, assume that  the principal curvature $\kappa_2$ is constant. Again, if   $\kappa_1$ is constant, then the surface is isoparametric and because $\Sigma$ is a $\lambda$-shrinker, then  $\Sigma$ is a plane, a sphere or a circular cylinder. This proves the result  in this situation.

Suppose now that $\kappa_1$ is not constant and we will arrive to a contradiction.

\begin{enumerate}
\item Case $\kappa_2=0$. Then $K=0$. From \eqref{b1} and   \eqref{d3}, we obtain $\gamma=0$, $\omega(e_2)=0$ and $e_1(\mu)=0$, $e_2(\mu)=1$.      In particular,  $\mu\not=0$ identically. The first equation of \eqref{d3} is 
$\mu\omega(e_1) +\frac{\lambda}{\alpha}\kappa_1=1$. From \eqref{b1} we know $\omega(e_1)=\frac{e_2(\kappa_1)}{\kappa_1}$, hence
\begin{equation}\label{con}
\frac{e_2(\kappa_1)}{\kappa_1}\mu+\frac{\lambda}{\alpha}\kappa_1=1.
\end{equation}
Differentiating with respect to $e_2$ and taking into account that $e_2(\mu)=1$, we have  
\begin{equation}\label{con2}
\frac{e_{22}(\kappa_1)}{\kappa_1}\mu+\frac{e_2(\kappa_1)}{\kappa_1}-\frac{e_2(\kappa_1)^2}{\kappa_1^2}\mu+\frac{\lambda}{\alpha}e_2(\kappa_1)=0.
\end{equation}
On the other hand, the identity \eqref{bb} is 
$$0=\frac{e_{22}(\kappa_1)}{\kappa_1}-2\frac{e_2(\kappa_1)^2}{\kappa_1^2}=0.$$ 
From this equation, we get the value of $e_{22}(\kappa_1)$ and putting in \eqref{con2}, we obtain 
$$\frac{e_2(\kappa_1)^2}{\kappa_1^2}\mu+\frac{e_2(\kappa_1)}{\kappa_1}+\frac{\lambda}{\alpha}e_2(\kappa_1)=0.$$
Because $e_2(\kappa_1)\not=0$, then
$$\frac{e_2(\kappa_1)}{\kappa_1^2}\mu+\frac{1}{\kappa_1}+\frac{\lambda}{\alpha}=0.$$
This gives a contradiction with \eqref{con}.
\item Case $\kappa_2\not=0$.    Let us work on the  open set $\Omega\subset \Sigma$ formed by non-umbilic points. From \eqref{b1} we have $\omega(e_2)=0$. By using the forth and sixth equations of \eqref{d3}, we deduce   $\mu=- e_2(\kappa_1)/\alpha$ and $e_2(\mu)=1+\frac{K-\lambda}{\alpha}\kappa_2$. From the first identity, we have  
$e_2(\mu)=-  e_{22}(\kappa_1)/\alpha$. Then 
\begin{equation}\label{d22}
e_{22}(\kappa_1)=-(K-\lambda)\kappa_2-\alpha.
\end{equation}
Using this identity, and because $\kappa_2$ is constant, equation \eqref{bb} is 
$$\kappa_1\kappa_2=\frac{e_{22}(\kappa_1)}{\kappa_1-\kappa_2}-2\frac{e_2(\kappa_1)^2}{(\kappa_1-\kappa_2)^2}=
-\frac{(K-\lambda)\kappa_2+\alpha}{\kappa_1-\kappa_2}-2\frac{e_2(\kappa_1)^2}{(\kappa_1-\kappa_2)^2}.$$
 Then 
 $$2e_2(\kappa_1)^2=-(\kappa_1-\kappa_2)(\kappa_1 K+\alpha-\lambda\kappa_2).$$
 Differentiating with respect to $e_2$, we have 
 \begin{equation}\label{fp22}
4e_2(\kappa_1)e_{22}(\kappa_1)=-2e_2(\kappa_1)\left(   \alpha -4 H^2 \kappa_1+9 H \kappa_1^2-3 H \lambda -4 \kappa_1^3+2 \lambda  \kappa_1 \right).
\end{equation}
 The function  $e_2(\kappa_1)$ is not zero because otherwise, we have $\mu=0$ and the forth equation in \eqref{d3} yields $K-\lambda=1$. Then  $K$ is constant and  both principal curvatures are constant because $\kappa_2\not=0$. This is a contradiction because this situation was initiated discarded.

 We insert  \eqref{d22} in \eqref{fp22}. Using that $\kappa_2\not=0$, we obtain  a polynomial on the function $\kappa_1$, namely, 
 $$-6 \kappa_1^3+17H\kappa_1^2-12H^2\kappa_1+\lambda H-\alpha =0.$$
  This gives a contradiction because $\kappa_1$ is a non-constant function. We conclude that the case $\kappa_2\not=0$ cannot occur.
 \end{enumerate}
  \end{proof}

\section{Proof of Theorem \ref{t2}}\label{s6}

We   prove Thm. \ref{t2}. If one principal curvature is constant, then the other one is also constant because $H$ is constant. Thus $\Sigma$ is a plane, a sphere or a circular cylinder.  This proves the theorem in this situation. 

From now we suppose that both principal curvatures are not zero constant. We follow as in the proof of Thm. \ref{t1}. By working in a subdomain of $\Omega$, if necessary, we can assume $\kappa_1\not=0$ and $\kappa_1-H\not=0$ because the principal curvatures are not constant. We have 
\begin{equation*}
\left\{
\begin{split}
e_1(\frac{K}{\alpha})&=-\frac{2}{\alpha}(\kappa_1-H)e_1(\kappa_1),\\
e_2(\frac{K}{\alpha})&=-\frac{2}{\alpha}(\kappa_1-H)e_2(\kappa_1).
\end{split}\right.
\end{equation*}
Combining with the last two equations of \eqref{d3}, we deduce
\begin{equation}\label{gm}
\gamma=\frac{2(\kappa_1-H)}{\alpha\kappa_1}e_1(\kappa_1),\quad \mu=\frac{2(\kappa_1-H)}{\alpha(2H-\kappa_1)}e_2(\kappa_1).
\end{equation}
We use \eqref{c2} together \eqref{gm} to substitute in the  first four equations of \eqref{d3}. This gives 
  
\begin{equation}\label{d12}
\left\{
\begin{split}
e_1(\gamma)&=\frac{e_2(\kappa_1)^2}{\alpha(2H-\kappa_1)}+\frac{K-\lambda}{\alpha}\kappa_1+1\\
e_1(\mu)&=-\frac{e_1(\kappa_1)e_2(\kappa_1)}{\alpha\kappa_1}\\
e_2(\gamma)&=-\frac{e_1(\kappa_1)e_2(\kappa_1)}{\alpha(2H-\kappa_1)}\\
e_2(\mu)&=\frac{e_1(\kappa_1)^2}{\alpha\kappa_1}+\frac{K-\lambda}{\alpha}(2H-\kappa_1)+1.
\end{split}\right.
\end{equation}
We differentiate $\gamma$ and $\mu$ in \eqref{gm}, obtaining
\begin{equation}\label{d13}\left\{
\begin{split}
e_1(\gamma)&=\frac{2(\kappa_1-H)}{\alpha\kappa_1}e_{11}(\kappa_1)+\frac{2H}{\alpha \kappa_1^2}e_1(\kappa_1)^2 \\
e_1(\mu)&= \frac{2(\kappa_1-H)}{\alpha(2H-\kappa_1)}e_{12}(\kappa_1)+\frac{2H}{\alpha(2H-\kappa_1)^2}e_1(\kappa_1)e_2(\kappa_1)\\
e_2(\gamma)&= \frac{2(\kappa_1-H)}{\alpha \kappa_1}e_{21}(\kappa_1)+\frac{2H}{\alpha\kappa_1^2}e_1(\kappa_1)e_2(\kappa_1)\\
e_2(\mu)&= \frac{2(\kappa_1-H)}{\alpha(2H-\kappa_1)}e_{22}(\kappa_1)+\frac{2H}{\alpha(2H-\kappa_1)^2}e_2(\kappa_1)^2.
\end{split}\right.
\end{equation}

 Combining \eqref{d13} with \eqref{d12}, we obtain 

\begin{subequations}
 \begin{align}e_{11}(\kappa_1)&=-\frac{He_1(\kappa_1)^2}{\kappa_1(\kappa_1-H)}+\frac{\kappa_1 e_2(\kappa_1)^2}{2(\kappa_1-H)(2H-\kappa_1)}+\frac{((K-\lambda)\kappa_1+\alpha)\kappa_1}{2(\kappa_1-H)}\label{d7}\\
e_{12}(\kappa_1)&=-\frac{\kappa_1^2-2H\kappa_1+4H^2}{2\kappa_1(\kappa_1-H)(2H-\kappa_1)}e_1(\kappa_1)e_2(\kappa_1)\label{d9}\\
e_{22}(\kappa_1)&=\frac{(2H-\kappa_1)e_1(\kappa_1)^2}{2\kappa_1(\kappa_1-H)}-\frac{He_2(\kappa_1)^2}{(\kappa_1-H)(2H-\kappa_1)}+\frac{\left((K-\lambda)(2H-\kappa_1)+\alpha\right)(2H-\kappa_1)}{2(\kappa_1-H)}.\label{d10}
\end{align}
\end{subequations}
We use that  $\kappa_1$ cannot be constant in an open set of $\Sigma$. Without loss of generality, we can assume that $e_2(\kappa_1)\not=0$ in $\Omega$. Two possibilities are discussed.
\begin{enumerate} 
\item Case  $e_1(\kappa_1)=0$ identically in $\Omega$. Then \eqref{d7}  implies
$$e_2(\kappa_1)^2=-  (2H-\kappa_1)(\alpha+(K-\lambda)\kappa_1).$$
Then \eqref{d10} is 
$$e_{22}(\kappa_1)=\frac{\alpha(2H-\kappa_1) +2H(\alpha+\kappa_1(K-\lambda))+(2H-\kappa_1)^2(K-\lambda)}{2 (\kappa_1-H)}.$$
Differentiating $e_2(\kappa_1)^2$, and because $e_2(\kappa_1)\not=0$, we obtain 
$$e_{22}(\kappa_1)=\frac{1}{2}\left(\alpha-2 (H-\kappa_1) \left(4 H \kappa_1-2 \kappa_1^2-\lambda \right)\right).$$
Equating the above two expressions of $e_{22}(\kappa_1)$, we obtain a polynomial on $\kappa_1$, which it is
 $$-3\kappa_1^4+12H \kappa_1^3-(12H^2+\lambda)\kappa_1^2+2(\alpha+H\lambda)\kappa_1+H(2H\lambda-5\alpha)=0.$$
 This polynomial on $\kappa_1$ has constant coefficients. Thus $\kappa_1$ is constant. This gives  a contradiction.

\item Case  $e_1(\kappa_1)\not=0$   in $\Omega$.  Substituting \eqref{d7} and \eqref{d10} into \eqref{c11}, we have
\begin{equation}\label{d14}
    2e_1(\kappa_1)^2+2e_2(\kappa_1)^2+R_1=0,
\end{equation}
where
$$R_1=(2H^2-2H\kappa+\kappa^2)\lambda-4 H^2 \kappa_1^2+4 H \kappa_1^3-\alpha H-\kappa_1^4.$$
Differentiating \eqref{d14} with respect to $e_1$, we have
$$
4e_1(\kappa_1)e_{11}(\kappa_1)+4e_2(\kappa_1)e_{12}(\kappa_1)+e_1(R_1)=0.
$$
By using \eqref{d7} and \eqref{d9}, and simplifying by $e_1(\kappa_1)$,  the above equation   can be rewritten as
\begin{equation}\label{d16}
2H e_1(\kappa_1)^2+2H e_2(\kappa_1)^2+R_2=0,
\end{equation}
where
$$R_2=\kappa_1 \left(\kappa_1 \left(-4 H^3+10 H^2 \kappa_1-10 H \kappa_1^2+3 \kappa_1^3-\alpha\right)-H \lambda  (H-2 \kappa_1)\right).$$
Combining \eqref{d14}  and \eqref{d16}, we obtain  
$$
3 \kappa_1^4-6 H \kappa_1^3+H \lambda  (2 H+\kappa_1) -\kappa_1\alpha -H\alpha=0.$$ 
We conclude that  $\kappa_1$ is constant, which it is not possible. This completes   the proof of  Thm. \ref{t2}.

\end{enumerate}

The above computations allow to characterize the $\lambda$-shrinkers with constant Gauss curvature.

\begin{theorem} Planes, spheres are circular cylinders are the only  $\lambda$-shrinkers of the GCF with constant Gaussian curvature.
\end{theorem}

\begin{proof}

Regarding if the principal curvature are constant, we discuss two cases.
\begin{enumerate}
\item A principal curvature is constant. Then Thm. \ref{t-d} proves the result.
\item Both principal curvatures are not constant.   The last two equations of \eqref{d3} imply $\gamma=\mu=0$.   Then the first equation of \eqref{d3} implies  
$$\frac{K-\lambda}{\alpha}\kappa_1=-1.$$
Therefore $K\not=\lambda$ and $\kappa_1$ is a constant. This gives a contradiction. This case cannot occur. 
\end{enumerate}
\end{proof}


\section*{Acknowledgements}
The author  has been partially supported by MINECO/MICINN/FEDER grant no. PID2023-150727NB-I00,  and by the ``Mar\'{\i}a de Maeztu'' Excellence Unit IMAG, reference CEX2020-001105- M, funded by MCINN/AEI/10.13039/ 501100011033/ CEX2020-001105-M.

\end{document}